\documentclass[11pt]{article} 
\usepackage{amsmath,amsthm} 
\usepackage{amssymb,latexsym}
\usepackage[T1]{fontenc}
\usepackage{authblk}
\usepackage{lipsum}
\usepackage{hyperref}
\usepackage{lipsum} 
\usepackage{lineno}
\usepackage{upgreek}
\usepackage{marvosym}
\modulolinenumbers[1]
\begin{document}
	\newtheorem{theorem}{Theorem}[section]
	\newtheorem{question}{Question}[section]
	\newtheorem{thm}[theorem]{Theorem}
	\newtheorem{lem}[theorem]{Lemma}
	\newtheorem{eg}[theorem]{Example}
	\newtheorem{prop}[theorem]{Proposition}
	\newtheorem{cor}[theorem]{Corollary}
	\newtheorem{rem}[theorem]{Remark}
	\newtheorem{deff}[theorem]{Definition}
	\numberwithin{equation}{section}
	\title{Nonnil-$S$-Laskerian Rings}
	
	\author[1]{Tushar Singh}
	\author[2]{Ajim Uddin Ansari}
	\author[3]{Shiv Datt Kumar}

	\affil[1, 3]{\small Department of Mathematics, Motilal Nehru National Institute of Technology Allahabad, Prayagraj 211004, India \vskip0.01in Emails: sjstusharsingh0019@gmail.com, tushar.2021rma11@mnnit.ac.in, sdt@mnnit.ac.in}
	\vskip0.05in
	\affil[2]{\small Department of Mathematics, CMP Degree College, University of Allahabad, Prayagraj-211002, India \vskip0.01in Email: ajimmatau@gmail.com}
	\maketitle
	\hrule
	
	\begin{abstract}
		\noindent
	In this paper, we introduce the concept of nonnil-$S$-Laskerian rings, which generalize both nonnil-Laskerian rings and $S$-Laskerian rings. A ring $R$ is said to be nonnil-$S$-Laskerian if every nonnil ideal $I$ (disjoint from $S$) of $R$ is $S$-decomposable. As a main result, we prove that the class of nonnil-$S$-Noetherian rings belongs to the class of nonnil-S-Laskerian rings. Also,	we prove that a nonnil-$S$-Laskerian ring has $S$-Noetherian spectrum under a mild condition.	Among other results, we prove that if the power series ring $R[[X]]$ is nonnil-$S$-Laskerian with
	$S$-decomposable nilradical, then $R$ is $S$-laskerian and satisfies the $S$-SFT property.

	\end{abstract}

	\textbf{Keywords:} Nonnil-$S$-Noetherian rings, nonnil-$S$-Laskerian rings, Laskerian rings,  $S$-Laskerian rings.\\
	\textbf{MSC(2020):} 13E05, 13E99, 13F25, 16N40.
	\hrule
	
	\section{Introduction}
	Throughout the paper, let $R$ be a commutative ring with identity, and $S$ be a multiplicative closed subset of $R$. We denote $Nil(R)$ by the set of nilpotent elements of $R$  called nilradical of $R$. An ideal $I$ of $R$ is said to be a nonnil ideal if $I\nsubseteq Nil(R)$. An ideal $I$ of $R$ is said to be divided if $I\subset (a)$, for every $a\in R\setminus I$.  In 2003, Badawi \cite{ab03} introduced the concept of nonnil-Noetherian rings as a generalization of Noetherian rings. A ring $R$ is said to be a nonnil-Noetherian ring if each nonnil ideal of $R$ is finitely generated. They transfered several well-known results of Noetherian rings to this class of rings. In 2020, Kwon and Lim \cite{mj20} extended the concept of nonnil-Noetherian rings to nonnil-$S$-Noetherian rings. A ring $R$ is said to be a nonnil-$S$-Noetherian ring if every nonnil ideal of $R$ is $S$-finite. 
	
	Theory of Laskerian rings, initiated by Lasker-Noether (\cite{el05}, \cite{ne21}), is one of the basic tools for both  commutative algebra and algebraic geometry. A ring $R$ is called Laskerian if every ideal of $R$ can be written as a finite intersection of primary ideals. It gives an algebraic foundation for decomposing an algebraic variety into its irreducible components. It is proved in historical Lasker-Noether theorem that every Noetherian ring is Laskerian. Due to their significance, numerous authors attempted to generalize the concept of Laskerian rings (see  \cite{rg80},  \cite{mj20}, \cite{sm22}, \cite{ts23}, \cite{ts24}, and \cite{vs22}). 	For instance, Subramanian \cite{vs22} introduced the concept of $S$-Laskerian rings, and Moulahi \cite{sm22} introduced the concept of nonnil-Laskerian rings as proper generalizations of Laskerian rings. A ring $R$ is called nonnil-Laskerian (respectively, $S$-Laskerian) if every nonnil ideal (respectively, every ideal disjoint from $S$) of $R$ can be written as a finite intersection of primary ideals (respectively, $S$-primary ideals). Hizem \cite{sh08} showed that the class of nonnil-Noetherian rings belongs to the class of nonnil-Laskerian rings.

	In this paper, we introduce the concept of nonnil-$S$-Laskerian rings, which generalize both nonnil-Laskerian rings and $S$-Laskerian rings. We provide an example of a nonnil-$S$-Laskerian ring which is not nonnil-Laskerian (see Example \ref{laskerian}). We extend several properties and characterizations of both  nonnil-Laskerian rings and $S$-Laskerian rings to this new class of rings. Among other results, Gilmer and Heinzer \cite{rg80} proved that every Laskerian ring has a Noetherian spectrum. We extend this result for the case of nonnil-$S$-Laskerian rings (see Theorem~\ref{spectrum}). It is well known that every $S$-primary ideal is $S$-decomposable but converse is not true in general. We provide a sufficient condition for an $S$-decomposable ideal to be $S$-primary (see Proposition \ref{primed}). Further, Eljeri \cite{em18} introduced the concept of the $S$-strong finite type rings (in short, $S$-SFT rings) as a generalization of SFT rings. We provide a connection between nonnil-$S$-Laskerian rings and $S$-SFT rings. More precisely, we prove that if a power series ring $R[[X]]$ is nonnil-$S$-Laskerian with $Nil(R[[X]])$ being $S$-decomposable, then $R$ is an $S$-SFT ring (see Proposition~\ref{psft}). Finally, as a main result of the paper, we prove that every nonnil-$S$-Noetherian ring is a nonnil-$S$-Laskerian ring (see Theorem~\ref{on}).
	
	\section{Main Results}
	\noindent
Recall \cite[Definition 2.1]{me22}, a proper ideal $Q$ of a ring $R$ disjoint from $S$ is said to be $S$-primary if there exists an $s \in S$ such that for all $a, b \in R$, if $ab \in Q$, then either $sa \in Q$ or $sb \in rad(Q)$. By \cite[Proposition 2.5]{me22}, if $Q$ is an $S$-primary ideal of $R$, then $P=\text{rad}(Q)$ is an $S$-prime ideal. In such a case, we say that $Q$ is a $P$-$S$-primary ideal of $R$. For an ideal $I$ of a ring $R$, we denote $S(I)=\{a\in R \hspace{0.2cm}| \hspace{0.2cm}\dfrac{a}{1}\in S^{-1}I\}$ called the contraction of $I$ with respect to $S$. We begin this section by introducing the concepts of $S$-decomposable ideals and nonnil-$S$-Laskerian rings.

	\begin{deff}\label{dec}
		\noindent
		Let $R$ be a ring and let $S$ be a multiplicative set of $R$. Let
		$I$ be an ideal of $R$ such that $I\cap S=\emptyset$. We say that $I$ admits $S$-primary decomposition if $I$ is a finite intersection of $S$-primary ideals of $R$. In such a case, we say that $I$ is $S$-decomposable. An $S$-primary decomposition $I=\bigcap\limits_{i=1}^{n}Q_{i}$ of $I$ with $rad(Q_{i})=P_{i}$ for each $i\in\{1, 2,\ldots, n\}$ is said to be minimal if the following conditions hold:
		\leavevmode
		\begin{enumerate}
			\item $S(P_{i})\neq S(P_{j})$ for all distinct $i, j\in\{1, 2, \ldots, n\}$.
			\item $S(Q_{i})\nsupseteq\bigcap_{j\in\{1, 2, \ldots, n\}\setminus\{i\}}S(Q_{j})$ for each $i\in \{1,  2, \ldots, n\}$ (equivalently,\\
			$S(Q_{i})\nsupseteq \bigcap_{j\in\{1, 2, \ldots, n\}\setminus\{i\}}Q_{j}$ for each $i\in \{1,  2, \ldots, n\})$. 
		\end{enumerate}
	\end{deff}
	
	\noindent
	Clearly, the concepts of $S$-primary decomposition and primary decomposition coincide for $S=\{1\}$. The following example shows that the concept of $S$-primary decomposition is a proper generalization of the concept of primary decomposition.
	\begin{eg}\label{bool}
		\noindent
		Consider the boolean ring $R=\mathbb{Z}_{2}\times \mathbb{Z}_{2}\times\cdots\times\mathbb{Z}_{2}\times\cdots$ (countably infinite copies of  $\mathbb{Z}_{2})$. According to \cite[Theorem 1]{vd17}, the zero ideal $(0)=(0,0,0,\ldots)$ in $R$ has no primary decomposition. Next, we show that $(0)$ is an  $S$-primary ideal of $R$. First, we observe that $(0)\cap S=\emptyset$. Now, let $a=(a_n)_{n\in\mathbb{N}}$, $b=(b_n)_{n\in\mathbb{N}} \in R$ such that $ab=0$, where each $a_{i},  {b_{i}}\in \mathbb{Z}_{2}$. This implies that $a_nb_n=0$ for all $n\in \mathbb{N}$, in particular, $a_1b_1=0$. Then we have either $a_1=0$ or $b_1=0$. If $a_1=0$, then $sa=0$.  If $b_1=0$, then $sb=0$. Thus $(0)$ is an $S$-primary ideal, and so is $S$-primary decomposable.
	\end{eg}
	

	\begin{prop}\label{ry}
		Let $S$ be a multiplicative set of a ring $R$. Then the following statements hold:
		\begin{enumerate}
			\item Finite intersection of $P$-$S$-primary ideals is $P$-$S$-primary.
			\item If $Q$ is a $P$-primary ideal of $R$ with $Q\cap S=\emptyset$, then for any ideal $J$ of $R$ with $J\cap S\neq\emptyset$, $Q \cap J$ is a $(P\cap rad(J))$-$S$-primary ideal of $R$.
		\end{enumerate}
	\end{prop}
	\begin{proof}
		\leavevmode
		\begin{enumerate}
			\item Let $Q_1,Q_2,\ldots,Q_n$ be $P$-$S$-primary ideals, then $S\cap Q_{i} =\emptyset$ for each $i=1,2,\ldots,n$, and so  $S\cap(\bigcap\limits_{i=1}^{n}Q_{i}) =\emptyset$. Suppose $Q=\bigcap\limits_{i=1}^{n}Q_{i}$. Since each $Q_{i}$ is $P$-$S$-primary, $rad(Q)=rad(\bigcap\limits_{i=1}^{n}Q_{i})=\bigcap\limits_{i=1}^{n}rad(Q_{i})=P$. Now, let $xy\in Q$,  where $x,y\in R$ and with $sy\notin Q$ for all $s\in S$. Consequently, for every $s\in S$, there exists $k_{s}$ such that $xy\in Q_{k_{s}}$ and $sy\notin Q_{k_{s}}$. Let $s_{i}\in S$ be the element satisfying the $S$-primary property for $Q_{i}$. Since we have finitely many $Q_{i}$, put $s=s_{1}s_{2}\dots s_{n}\in S$. Now, fix $s$ and assume that $xy\in Q$ but $sy\notin Q$. Thus, there exists $k$ such that $xy\in Q_{k}$ and $sy\notin Q_{k}$. Then for $s_{k}\in S$, we obtain $s_{k}x\in rad(Q_k)=P$ or $s_{k}y\in Q_{k}$. The latter case gives $sy\in Q_{k}$, a contradiction. Thus $sx\in rad(Q)=P$, and therefore $Q$ is $P$-$S$-primary.
			\item As $Q\cap S=\emptyset$, it follows that $(Q\cap J)\cap S=\emptyset$. By assumption, $J\cap S\neq\emptyset$.
			Let $s \in J\cap S$. Let $a, b\in R$ be such that $ab\in Q \cap J$. Either $a\in Q$ or $b\in rad(Q) = P$ since $Q$ is a $P$-primary ideal of $R$. Hence either $sa\in Q \cap J$ or $sb \in P\cap J \subseteq P\cap rad(J) = rad(Q)\cap rad(J) = rad(Q \cap J)$. This proves that $Q\cap J$ is a $(P\cap rad(J))$-$S$-primary ideal of $R$.
		\end{enumerate}
	\end{proof}
	
	\begin{rem}\label{11}
		\noindent
		Let $R$ be a ring, and let $S$ be a multiplicative closed subset of $R$. Let $I$ be an ideal of $R$ such that $I\cap S=\emptyset$. Suppose that $I$ admits an $S$-primary decomposition. Let $I=\bigcap\limits_{i=1}^{n}Q_{i}$ be an $S$-primary decomposition of $I$ with $Q_{i}$ is a $P_{i}$-$S$-primary ideal of $R$ for each $i\in\{1,2,\ldots, n\}$. Then, by \cite[Proposition 2.7]{me22}, $S^{-1}Q_{i}$ is $S^{-1}P_{i}$-primary, $S(I)=\bigcap\limits_{i=1}^{n}S(Q_{i})$ is a primary decomposition with $S(Q_{i})$ is a $S(P_{i})$-primary for each $i\in\{1,\ldots, n\}$. Let $1\leq i\leq n$. Let $s_{i}\in S$ be such that it satisfies the $S$-primary property of $Q_{i}$. Notice that $(P_{i}:s_{i})$ is a prime ideal of $R$ and $(Q_{i}:s_{i})$ is a $(P_{i}:s_{i})$-primary ideal of $R$. Let $s=\prod_{i=1}^{n}s_{i}$. Then $s\in S$. Observe that $(Q_{i}:s_{i})=(Q_{i}:s)$, $(P_{i}:s_{i})=(P_{i}:s)$, $S(Q_{i})=(Q_{i}:s)$, and $S(P_{i})=(P_{i}:s)$. From $I=\bigcap\limits_{i=1}^{n}Q_{i}$, it follows that $(I:s)=\bigcap\limits_{i=1}^{n}(Q_{i}:s)=\bigcap\limits_{i=1}^{n}S(Q_{i})=S(I)$. Let $k$ of the $S(P_1), \ldots, S(P_n)$ be distinct. After a suitable rearrangement of $\{1, \ldots, n\}$, we can assume without loss of generality that $S(P_{1}),\ldots, S(P_k)$ are distinct among $S(P_{1}),\ldots, S(P_n)$. Let $A_{1}=\{j\in\{1, \ldots, n\}\mid S(Q_{j})~ is ~S(P_{1})$-primary$\}$, \ldots, $A_{k}=\{j\in\{1, \ldots, n\}\mid S(Q_{j})~ is ~S(P_{k})$-primary $\}$. 
		
		It is clear that $1 \in A_1,\ldots, k\in A_{k}$, and $\{1,\ldots,n\}=\bigcup\limits_{t=1}^{k}A_{t}$. Let $1\leq t\leq k$. Notice that $\bigcap_{j\in A_{t}}S(Q_{j})$ is $S(P_{t})$-primary by \cite[Lemma 4.3]{fm69}. It is convenient to denote $\bigcap_{j\in A_{t}}Q_{j}$ by $I'_{t}$. Thus $(I:s)=S(I)=S(I'_{1})\cap S(I'_{2})\cap \cdots\cap S(I'_{k})$ with $S(I'_{t})$ is $S(P_{t})$-primary for each $t\in \{1,\ldots, k\}$ and $S(P_1),\ldots, S(P_k)$ are distinct. After omitting those $S(I'_{i})$ such that $S(I'_{i})\supseteq \bigcap_{t\in\{1,\ldots, k\}\setminus\{i\}}S(I'_{t})$ from the intersection, we can assume  without loss of generality that $(I:s)=S(I)=\bigcap\limits_{t=1}^{k}S(I'_{t})$ is a minimal primary decomposition of $S(I)$. Next, we claim that $I=(I:s)\cap (I+Rs)$. It is clear that $I\subseteq (I:s)\cap (I+Rs)$. Let $y\in(I:s)\cap (I+Rs)$. Then $ys\in I$ and $y=a+rs$ for some $r\in R$. This implies that $ys=as+rs^{2}$ and so, $rs^{2}\in I$. Hence, $r\in S(I)=(I:s)$. Therefore, $y=a+rs\in I$. This shows that $(I:s)\cap (I+Rs)\subseteq I$. Thus $I=(I:s)\cap (I+Rs)$ and hence, $I=(\bigcap\limits_{t=1}^{k}S(I'_{t})\cap (I+Rs))=\bigcap\limits_{t=1}^{k}(S(I'_{t})\cap (I+Rs))$. Let $1\leq t\leq k$.  For convenience, let us denote $S(I'_{t})\cap (I+Rs)$ by $Q'_{t}$. As $S(I'_{t})$ is $S(P_{t})$-primary with $S(I'_{t})\cap S=\emptyset$ and $(I+Rs)\cap S\neq \emptyset$, we obtain from Proposition \ref{ry}(2) that $Q'_{t}$ is $S(P_{t})\cap rad(I+Rs)$-$S$-primary. Since $S(S(P_{t}))=S(P_{t})$, $S(rad(I+Rs))=R$, it follows that $S(S(P_{t})\cap rad(I+Rs))=S(P_{t})$. Notice that $S(Q'_{t})=S(S(I'_{t})\cap (I+Rs))=S(I'_{t})$, as $S(S(I'_{t}))=S(I'_{t})$ and $S(I+Rs)=R$. Hence, for all distinct $i, j\in\{1,\ldots, k\}$, $S(S(P_{i})\cap rad(I+Rs))\neq S(S(P_{j})\cap rad(I+Rs))$ and for each $i$ with $1\leq i\leq t$, $S(Q'_{i})\nsupseteq \bigcap_{t\in\{1,\ldots, k\}\setminus\{i\}}S(Q'_{t})$. Therefore, $I=\bigcap\limits_{t=1}^{k}Q'_{t}$ is a minimal $S$-primary decomposition.
		
	\end{rem}

	\begin{deff}
		A  ring $R$ is said to be a nonnil-$S$-Laskerian ring if every nonnil-ideal (disjoint from $S$) of $R$  is $S$-decomposable.
	\end{deff}

		Let $S_1 \subseteq S_2$ be two multiplicative subsets of a ring $R$. If $R$ is a nonnil-$S_1$-Laskerian ring, then $R$ is also a  nonnil-$S_2$-Laskerian ring. If $R$ is a reduced ring, the notion of  nonnil-$S$-Laskerian rings coincides exactly with that of $S$-Laskerian rings.
		
		 Recall from \cite{ah18} that a ring $R$ is said to have a \textit{Noetherian spectrum} if $R$ satisfies the ascending chain condition (ACC) on radical ideals. This is equivalent to the condition that $R$ satisfies the ACC on prime ideals, and each ideal has only finitely many prime ideals minimal over it. 
		 
		 Clearly, every nonnil-Laskerian ring is a nonnil-$S$-Laskerian ring. However, the following example shows that a nonnil-$S$-Laskerian ring may not be a nonnil-Laskerian ring. 
	
	\begin{eg}\label{laskerian}
		Let  $R=F[x_1,x_2,\ldots, x_{n}, \ldots]$ be the polynomial ring  in infinitely many indeterminates over a field  $F$.  Since $R$ has an ascending chain of prime ideals $(x_1)\subseteq (x_1,  x_2)\subseteq \cdots\subseteq (x_1,x_2,\ldots, x_n)\subseteq\cdots$ which does not terminate, so $R$ has no Noetherian spectrum. Evidently, $Nil(R)=(0)$ is the primary ideal so it is decomposable. This implies that $R$ is  not a nonnil-Laskerian ring since every nonnil-Laskerian ring  with decomposable nilradical $Nil(R)$ has a Noetherian spectrum, by \cite[Proposition 2.11]{sm22}. Consider the multiplicative set $S=R\setminus\{0\}$. Since $R$ has no nonnil ideal which is disjoint from $S$, and hence $R$ is a nonnil-$S$-Laskerian ring but not a nonnil-Laskerian.
	\end{eg}
%
	\begin{prop}
		Let $R$ be a nonnil-$S$-Laskerian ring and $I$ an ideal of $R$. If $I\subset Nil(R)$, then the quotient ring $R/I$ is nonnil-$\overline{S}$-Laskerian, where $\overline{S}=\{s+I\mid s\in S\}$ is a multiplicative closed subset of $R/I$. 
	\end{prop}
	\begin{proof}
		Let $y+I\in Nil(R/I)$. Then there exists $k\in\mathbb{N}$ such that $y^{k}+I=I$. This implies that $y^{k}\in I\subset Nil(R)$, and so $y\in rad(Nil(R))=Nil(R)$. Thus $Nil(R/I)\subseteq Nil(R)/I$. For reverse containment, let $\alpha +I\in Nil(R)/I$, where $\alpha\in Nil(R)$. Then there exists $n\in\mathbb{N}$ such that $\alpha^{n}=0$. This implies that $\alpha^{n}+I=I$, and so $(\alpha+I)^{n}=I$. Thus $\alpha+I\in Nil(R/I)$, and therefore $Nil(R/I)=Nil(R)/I$. Now, let $J$ be a proper nonnil ideal (disjoint from $\overline{S}$) of $R/I$. Write $J=Q'/I$, where $Q'$ is a proper ideal of $R$ containing $I$. Since $J\nsubseteq Nil(R/I)=Nil(R)/I$, $Q'\nsubseteq Nil(R)$. This implies that $Q'$ is a nonnil ideal of $R$. Now, we show that $Q'$ is disjoint from $S$. If $Q'\cap S\neq\emptyset$, then there exists $s\in S$ such that $s\in Q'$. This implies that $s+I\in Q'/I=J$, a contradiction, as $J$ is disjoint from $\overline{S}$. Thus $Q'$ is a nonnil ideal (disjoint from $S$) of \( R \). Since $R$ is nonnil-$S$-Laskerian, $Q'$ is $S$-decomposable, i.e., $Q'=\bigcap\limits_{i=1}^{n}Q_{i}$, where $Q_{i}$ $(1\leq i\leq n)$ is $P_{i}$-$S$-primary. Consequently, $J=\left(\bigcap\limits_{i=1}^{n}Q_{i}\right)/I=\bigcap\limits_{i=1}^{n}(Q_{i}/I)$. By \cite[Proposition 2.10]{me22}, each $Q_{i}/I$ is $\overline{S}$-primary. Thus $J$ is $\overline{S}$-decomposable, and so $R/Nil(R)$ is nonnil-$\overline{S}$-Laskerian.
	\end{proof}
	
	\begin{rem}
		Let $R$ be a nonnil-$S$-Laskerian ring. Then $S^{-1}R$ is a nonnil-Laskerian ring.  To show this, let \( J \) be a proper nonnil ideal of \( S^{-1}R \). Then $J = S^{-1}I$ for some proper nonnil ideal \( I \) of \( R \) which is disjoint from \( S \). Evidently, $I$ is $S$-decomposable. Write $I = \bigcap\limits_{i=1}^{n}Q_{i}$, where each \(Q_{i} \) is \( S \)-primary. Consequently, $J=S^{-1}I = \bigcap\limits_{i=1}^{n}S^{-1}Q_{i}$. By \cite[Proposition 2.7(3)]{me22}, each \( S^{-1}Q_{i}\) is a primary ideal in $S^{-1}R$, and so $J$ is $S$-decomposable. 
	\end{rem}
	
	\begin{prop}\label{las}
		Let $R$ be a ring with $S$-decomposable nilradical. If R is  nonnil-$S$-Laskerian, then $R/Nil(R)$ is an $\overline{S}$-Laskerian ring, where $\overline{S}=\{s+Nil(R)\mid s\in S\}$ is a multiplicative closed subset of $R/Nil(R)$.  
	\end{prop}
	\begin{proof}
		Let $J$ be an ideal (disjoint from $\overline{S}$) of $R / \text{Nil}(R)$. If $J=0$, then it is $S$-decomposable since $Nil(R)$ is $S$-decomposable. Now, let $J\neq 0$. Then $J = I / \text{Nil}(R)$ for some proper ideal $I$ of $R$ containing ${Nil}(R)$. If there exists $s\in I\cap S$, then $s+Nil(R)\in I/Nil(R)=J$, a contradiction since $J$ is disjoint from $\overline{S}$. This indicates that $I$ is disjoint from $S$. Also, if $I\subseteq Nil(R)$, then $I=Nil(R)$, and so $J=0$, a contradiction. Since $R$ is nonnil-$S$-Laskerian, $I$ is $S$-decomposable, i.e., $I=\bigcap\limits_{i=1}^n Q_i $, where each $Q_{i}$ is $P_{i}$-$S$-primary. This concludes that $J = \left(\bigcap\limits_{i=1}^n Q_i \right)/ \text{Nil}(R)$. For each $i=1,\ldots, n$, by \cite[Proposition 2.10]{me22}, $Q_{i}/Nil(R)$ is $\overline{S}$-primary. Now, we show that $J = \bigcap\limits_{i=1}^n \left(Q_i / \text{Nil}(R)\right)$. It is easy to see that $J \subseteq \bigcap\limits_{i=1}^n \left(Q_i / \text{Nil}(R)\right)$. Now, let $y \in\bigcap\limits_{i=1}^n \left(Q_i / \text{Nil}(R)\right)$. If $y=0$, then $y\in J$. Suppose $y\neq 0$, then for all $i \in \{1, 2, \dots, n\}$, there exists an $x_i \in Q_i \setminus \text{Nil}(R)$ such that $y = x_i + \text{Nil}(R)$. Thus for all $i \in \{2, \dots, n\}$, there exists a $w_i \in \text{Nil}(R)$ such that $x_1 = w_i + x_i$. But $w_i \in \text{Nil}(R) \subset Q_i$, for all $i \in \{2, \dots, n\}$, so $x_1 \in \bigcap\limits_{i=1}^n Q_i$.  Hence $y=x_{1}+Nil(R)\in \left(\bigcap\limits_{i=1}^n Q_i\right)/Nil(R)=J$, as desired.
	\end{proof}

 Next, we prove the converse of Proposition~\ref{las} under the condition \( \text{Nil}(R) \) is divided.

	\begin{prop}\label{nil}
		Let $R$ be a ring with a divided nilradical. If $R/\text{Nil}(R)$ is an $\overline{S}$-Laskerian ring, where $\overline{S}=\{s+Nil(R)\mid s\in S\}$ is a multiplicative closed subset of $R/Nil(R)$, then $R$ is a  nonnil-$S$-Laskerian ring.
	\end{prop}
	\begin{proof}
		Let $I$ be a nonnil ideal of $R$ disjoint from $S$. Then there exists $y\in I\setminus Nil(R)$ such that $Nil(R)\subset yR\subseteq I$ since $Nil(R)$ is divided. This implies that $J = I/\text{Nil}(R)$ is an ideal of $R/Nil(R)$. If $J\cap\overline{S}\neq\emptyset$, then $s+Nil(R)=i+Nil(R)$ for some $s\in S$ and $i\in I$. Consequently,  $s-i\in Nil(R)\subseteq I$, and so $s\in I$, a contradiction as $I\cap S=\emptyset$. Thus $J$ is disjoint from $\overline{S}$. Now, since $R/\text{Nil}(R)$ is $\overline{S}$-Laskerian,  $J$ is $\overline{S}$-decomposable. Write $J = \bigcap_{i=1}^n (Q_i /\text{Nil}(R))=\left(\bigcap\limits_{i=1}^n Q_i\right)/Nil(R)$, where each $Q_i$ is a $P_{i}$-$S$-primary ideal of $R$. This implies that $I\subseteq  \bigcap\limits_{i=1}^n Q_i$. For reverse containment, let  $x\in  \bigcap\limits_{i=1}^n Q_i$. Then $x+Nil(R)\in J$, and so we can write  $x+Nil(R)=a+Nil(R)$ for some $a\in I$. Consequently, $x\in Nil(R)\subseteq I$. Therefore $I=\bigcap\limits_{i=1}^n Q_i$, as required. 
	\end{proof}

The condition  \textquotedblleft divided $Nil(R)$ \textquotedblright~ is necessary in Proposition \ref{nil}. For intance, let $R=\mathbb{Z}(+)\mathbb{Q}$ be the idealization of $\mathbb{Q}$ as a $\mathbb{Z}$-module and $A=R[X]$. Then, by \cite[Example 4.8]{me22}, $A=R[X]$ is not a nonnil-$S$-Laskerian ring  but $A/Nil(A)$ is an $\overline{S}$-Laskerian for $S=\{1\}$, and $Nil(A) = \left((0)(+)\mathbb{Q}\right)[X]$ is not a divided ideal. 

	\begin{prop}\label{primed}
		Let $R$ be a ring such that $Nil(R)$ is a prime divided ideal disjoint from $S$, and let $I$ be a nil ideal of $R$. Then $I$ is $S$-decomposable if and only if $I$ is $S$-primary. In particular, $R$ is $S$-Laskerian if and only if $R$ is  nonnil-$S$-Laskerian and every ideal contained in $Nil(R)$ is Nil(R)-$S$-primary.
	\end{prop}
	\begin{proof}
	Suppose $I$ is $S$-decomposable. Let $I=\bigcap\limits_{i=1}^{n} Q_{i}$ be a minimal $S$-primary decomposition of $I$, where each $Q_{i}$ is $P_{i}$-$S$-primary. If $\text{Nil}(R) = (0)$, then $I=Nil(R)=(0)$ since $I$ is a nil ideal. By hypothesis, $Nil(R)$ is a prime ideal, and hence $I$ is $S$-primary. Now, suppose $\text{Nil}(R)\neq (0)$. Since $I=\bigcap\limits_{i=1}^{n} Q_{i}$,  $\text{rad}(S(I))= S(\text{Nil}(R))=\bigcap\limits_{i=1}^{n} \text{rad}(S(Q_{i})) = \bigcap\limits_{i=1}^{n} S(P_{i})$. This implies that $S(\text{Nil}(R))=S(P_{i_{0}})$ for some $i_{0}$ as $Nil(R)$ is a prime ideal. Evidently, $Q_{i_{0}}$ is a nil ideal. Now, we claim that each $Q_{i}$ is a nil ideal. Contrary, suppose there exists \(j \in \{1, \ldots, n\} \setminus \{i_0\}\) such that $Q_j$ is nonnil. This implies that $Nil(R)\subset Q_{j}$, and so $S(\text{Nil}(R)) \subset S(Q_{j})$. Consequently, $S(Q_{i_{0}}) \subseteq S(P_{i_{0}}) = S(\text{Nil}(R)) \subset S(Q_{j})$, which contradicts the minimality of $S$-decomposition. Hence each $Q_i$ is a nil ideal of $R$. Now, we show that $Q_{j}=Q_{i_{0}}$ for each $j$. Suppose $Q_{i_{0}}\neq Q_{j}$ for some $j \neq i_{0}$. Since $Q_{i_{0}}$ and $Q_{j}$ are both contained in $\text{Nil}(R)$, we have  $S(P_{i_{0}}) = S(P_{j}) = S(Nil(R))$. This contradicts the first condition of Definition \ref{dec}(1). Therefore $I = Q_{i_{0}}$, i.e.,  $I$ is $S$-primary. Conversely, if $I$ is $S$-primary, then $I$ is $S$-decomposable. The second statement follows from the first.
	\end{proof}
	Recall from \cite{ah18}, an ideal $I$ of $R$ is called \textit{radically $S$-finite} if there exist $s\in S$ and a finitely generated ideal $J$ of $R$ such that $sI \subseteq \sqrt{J} \subseteq \sqrt{I}$. Every prime ideal of $R$ is radically $S$-finite if and only if  $R$ satisfies  the \emph{$S$-Noetherian spectrum property} (see \cite[Theorem 2.2.]{ah18}).
	
	 In \cite[Proposition 2.11]{sm22}, it was shown that if $R$ is a nonnil-Laskerian ring with decomposable nilradical $Nil(R)$, then $R$ has a Noetherian spectrum. In the following theorem, we extend this result for nonnil-$S$-Laskerian rings. 
	\begin{theorem}\label{spectrum}
		If $R$ is a  nonnil-$S$-Laskerian ring with $S$-decomposable nilradical $Nil(R)$, then it has $S$-Noetherian spectrum.
	\end{theorem}
	\begin{proof}
		Let $P$ be a nonzero prime ideal of $R$. If $Nil(R)$ is prime and $P \subseteq Nil(R)$, then $rad(0) = Nil(R)\subseteq P \subseteq Nil(R)$, which implies that $P = Nil(R) = rad(0)$. Hence $P$ is radically $S$-finite. Now, suppose $P$ is a nonnil prime ideal of $R$. By Proposition \ref{las}, $R/Nil(R)$ is $\overline{S}$-Laskerian, and so by \cite[Proposition 3.11]{vs22}, $R/Nil(R)$ has an $\overline{S}$-Noetherian spectrum. Consequently, the nonzero prime ideal $J = P/Nil(R)$ of $R/Nil(R)$ is radically $\overline{S}$-finite. This implies that there exist $\bar{s}\in\overline{S}$ and a finitely generated ideal $K$ of $R/Nil(R)$ such that  $\bar{s}J\subseteq rad(K)\subseteq J$. Suppose $K=(\bar{x}_{1},  \ldots, \bar{x}_{n})$, for some $\bar{x}_{1}, \ldots, \bar{x}_{n}\in R/Nil(R)$. Consider an ideal  $I=(x_{1},\ldots, x_{n})$ of $R$. Now, we prove $sP\subseteq rad(I)\subseteq P$ for some $s\in S$. For this, let $x\in P$. Then $\bar{x}\in J$, and so $\bar{s}\bar{x}\in \bar{s}J\subseteq rad(K)$. This implies that $(\bar{s}\bar{x})^{k}=\bar{a}_{1}\bar{x}_{1}+\cdots+\bar{a}_{n}\bar{x}_{n}$, for some $\bar{a}_{1}, \ldots, \bar{a}_{n}\in R/Nil(R)$ and $k\in\mathbb{N}$. Then there exists $y\in Nil(R)$ such that $(sx)^{k}=a_{1}x_{1}+\cdots+a_{n}x_{n}+y$. Let $n$ be a positive integer such that $y^{n}=0$. Then, $(sx)^{kn}\in I$ and therefore $sP\subseteq rad(I)$. Further, let $z\in rad(I)$. Then $z^{m}=\alpha_{1}x_{1}+\cdots+\alpha_{n}x_{n}$, for some $\alpha_{1},\ldots, \alpha_{n}\in R$ and $m\in\mathbb{N}$. Consequently, $ \bar{z}^{m}=\bar{\alpha}_{1}\bar{x}_{1}+\cdots+\bar{\alpha}_{n}\bar{x}_{n}\in K$, and so $\bar{z}\in rad(K)\subseteq J=P/Nil(R)$. This implies that $z\in P$, and thus $sP\subseteq rad(I)\subseteq P$, as desired.
	\end{proof}
	
	\begin{prop}\label{lask}
		Let $R$ be a nonnil-$S$-Laskerian ring. If $I$ is a nonnil ideal of $R$, then $R/I$ is an $\overline{S}$-Laskerian ring.
	\end{prop}
	\begin{proof}
		Let $R'=R/I$ and  $J$ be an ideal (disjoint from $\overline{S}$) of $R'$. Then $J=K/I$ for some ideal $K\supset I$ of $R$. If $K\subseteq Nil(R)$, then $I\subseteq Nil(R)$, a contradiction that $I$ is nonnil. Therefore $K$ is nonnil ideal of $R$. Also, $K$ is disjoint from $S$ because $J$ is disjoint from $\overline{S}$. Since $R$ is nonnil-$S$-Laskerian, $K=\bigcap\limits_{i=1}^{n}Q_{i}$, where each $Q_{i}$ is $P_{i}$-$S$-primary. This implies that $J=\left(\bigcap\limits_{i=1}^{n}Q_{i}\right)/I$. For each $i=1,\ldots, n$, by \cite[Proposition 2.10]{me22}, $Q_{i}/I$ is $\overline{S}$-primary. Now, we show that $J=\bigcap\limits_{i=1}^n \left(Q_i /I \right)$. It is easy to see that $J \subseteq \bigcap\limits_{i=1}^n \left(Q_i /I\right)$. Let now $y+I \in\bigcap\limits_{i=1}^n \left(Q_i /I\right)$. Then for each $i=1,\ldots, n$, there exists $x_{i}\in Q_{i}$  such that $y+I=x_{i}+I$. Consequently, there exists $a_i \in I$ for all $i \in \{i, \dots, n\}$ such that $y= a_i + x_i$. But $a_i \in I \subset Q_i$, for all $i \in \{1, \dots, n\}$. Hence $y\in \bigcap\limits_{i=1}^n Q_i$ and therefore $y+I\in J$. 
	\end{proof}
	
	Recall $\cite{ab22}$ that an ideal $I$ of $R$ is called a strong finite type (in short, SFT) ideal if there exist a finitely generated ideal $F\subseteq I$ of $R$ and $n\in\mathbb{N}$ such that for each $x\in I$, $x^n \in F$. Eljeri \cite{em18} extended the definition of $SFT$ to $S$-$SFT$ by using a multiplicative closed subset $S$ of $R$. An ideal $I$ of $R$ is called an $S$-strong finite type (in short $S$-SFT) ideal if there exist an $S$-finite ideal $F\subseteq I$ of $R$ and  $n\in\mathbb{N}$ such that for each $x\in I$, $x^n \in F$. Also, $R$ is said to be a $S$-strongly finite type ring (in short $S$-SFT ring ) if all its ideals are $S$-SFT ideals.\\


	\noindent
	Now, we provide an example of an $S$-SFT ideal, which is not $SFT$.

	\begin{eg}
		\noindent
		Consider $R=F[x_1,x_2,\ldots, x_{n}, \ldots]$, the polynomial ring in infinitely many indeterminates over a field $F$. Let $I=(x_1, x_2, \ldots, x_{n}, x_{n+1}, \ldots)$, which is not a finitely generated ideal of $R$. Evidently, $I$ is not an $SFT$ ideal because if it were $SFT$, then there would exist a finitely generated ideal $K$ of the form $(x_1, x_2, \ldots, x_n)$. However, for any $x = x_i \in I$ with $i > n$, we have $x^k \notin K$ for all $k \in \mathbb{N}$. 
		Now, consider the multiplicative set $S=R\setminus\{0\}$. By \cite[Proposition 2(a)]{ad02}, the ring $R$ is an $S$-Noetherian ring. Hence the ideal $I$ is $S$-$SFT$.
	\end{eg}
	
We denote $Spec(R, S)$ by the set of all prime ideals of $R$ disjoint from $S$. The next result provides insights into the formal power series ring in the context of being nonnil-\( S \)-Laskerian.
	\begin{prop}\label{psft}
		Let $R[[X]]$ be the formal power series ring  with $Nil(R[[X]])$ is $S$-decomposable. If $R[[X]]$ is nonnil-$S$-Laskerian, then $R$ has the $S$-SFT property.
	\end{prop}
	\begin{proof}
		For all $P \in Spec(R, S)$, $P[[X]]\in Spec(R[[X]], S)$, and then is $S$-radical of $R[[X]]$, by \cite[Remark 3.2(2)]{em18}. Suppose that $R$ is not an $S$-SFT ring. Then  there exists an ideal $P\in Spec(R, S)$ such that $P[[X]]$  is not a $S$-radically finite,  by \cite[Theorem 3.8]{em18}. Thus $R[[X]]$ does not have a $S$-Noetherian spectrum since a ring $R$ has $S$-Noetherian spectrum if each $S$-radical ideal of $R$ is $S$-radically finite, by \cite[Theorem 3.6(2)]{em18}, and therefore $R[[X]]$ is not nonnil-$S$-Laskerian.
	\end{proof}
	\begin{prop}
		Let $R[[X]]$ be a nonnil-$S$-Laskerian. Then $R$ is $S$-Laskerian.
	\end{prop}
	\begin{proof}
		Applying Proposition \ref{lask} for $R[[X]]$ and the nonnil ideal $XR[[X]]$.
		
	\end{proof}
		Recall \cite{ad02}, an ideal $I$ of $R$ is called $S$-finte if $sI\subseteq J\subseteq I$ for some finitely generated ideal $J$ of $R$ and some $s\in S$. Also, a ring $R$ is said to be an $S$-Noetherian ring if every ideal of $R$ is $S$-finite. Kwon and Lim \cite{mj20} extended the concept of $S$-Noetherian rings to nonnil-$S$-Noetherian rings.
		
	\begin{deff}\cite{ab22}
	A ring $R$ is called \textit{nonnil-$S$-Noetherian} if each nonnil ideal of $R$ is $S$-finite.
	\end{deff}
		 Form now, our aim is to prove that nonnil-$S$-Noetherian rings belong to the class of nonnil-$S$-Laskerian rings. To show this, we require an $S$-version of irreducible ideals.
	
	\begin{deff}\label{S-iir.}
		An ideal $Q$ (disjoint from $S$) of the ring $R$ is called  $S$-irreducible if $s(I\cap J)\subseteq Q \subseteq I\cap J$ for some $s\in S$ and some ideals $I$, $J$ of $R$, then there exists $s'\in S$ such that either $ss'I\subseteq Q$ or $ss'J\subseteq Q$.
	\end{deff}
	
	\noindent
	It is clear from the definition that every irreducible ideal is an $S$-irreducible ideal. However, the following example shows that an $S$-irreducibile ideal need not be irreducible.

	\begin{eg}\label{fm}
		\noindent
		Let $R=\mathbb{Z}$, $S=\mathbb{Z}\setminus 3\mathbb{Z}$ and $I=6\mathbb{Z}$. Since $I=2\mathbb{Z}\cap 3\mathbb{Z}$, therefore $I$ is not an irreducible ideal of $R$. Now, take $s=2\in S$. Then $2(3\mathbb{Z})=6\mathbb{Z}\subseteq I$. Thus $I$ is an $S$-irreducible ideal of $R$.
	\end{eg}
	
	\noindent
	Following \cite{zb17}, 
	let  $E$ be a family of ideals of a ring $R$. An element $I\in E$ is said to be an \textit{$S$-maximal element} of $E$ if there exists an $s\in S$ such that for each $J\in E$, if $I\subseteq J$, then $sJ\subseteq I$. Also a chain of ideals $(I_{i})_{i\in\wedge}$  of $R$ is called \textit{$S$-stationary} if there exist $k\in \wedge$ and $s\in S$ such that  $sI_{i}\subseteq I_{k}$ for all $i\in\wedge$, where $\wedge$ is an arbitrary indexing set. A family $\mathcal{F}$ of ideals of a ring $R$ is said to be $S$-saturated if it satisfies the following property: for every ideal $I$ of $R$, if there exist $s\in S$ and $J\in\mathcal{F}$ such that $sI\subseteq J$, then $I\in\mathcal{F}$.

	\begin{lem}\label{nonnil} The following assertions are equivalent for a ring $R$:
		\begin{enumerate}
			\item[(a)] The ring $R$ is nonnil-$S$-Noetherian.
			\item[(b)] Every ascending chain of nonnil ideals of $R$ is $S$-stationary.
			\item[(c)] Every nonempty $S$-saturated set of nonnil ideals of $R$ has a maximal element.
			\item[(d)] Each nonempty set of nonnil ideals of $R$ has a $S$-maximal element with respect to inclusion.
		\end{enumerate}
	\end{lem}
	
	\begin{proof}
		\leavevmode	
		\begin{enumerate}
			\item $(a)\implies(b)$. Let $(I_n)_{n\in\wedge}$ be an increasing sequence of nonnil ideals of $R$. Define the ideal $I = \bigcup\limits_{n\in\wedge}I_n$. If $I\subseteq Nil(R)$, then $I_{n}\subseteq Nil(R)$ for all $n$, which is not possible since each $I_n$ is nonnil. Thus $I$ is a nonnil ideal of $R$. Also, $I$ is $S$-finite since $R$ is nonnil-$S$-Noetherian. Consequently, there exist a finitely  generated ideal $J\subseteq R$ and $s\in S$ such that $sI\subseteq J\subseteq I$. Since $J$ is finitely generated, there is a $k\in\wedge$ satisfying $J\subseteq I_{k}$. Then we have $sI\subseteq J\subseteq I_{k}$, from which it follows that $sI_n \subseteq I_{k}$ for each $n\in\wedge$.
			
			\item $(b)\implies(c)$. Let $\mathcal{D}$ be an $S$-saturated set of nonnil ideals of $R$. Then given any chain \(\{I_n\}_{n \in\wedge} \subseteq \mathcal{D} \), we claim that \( I = \bigcup\limits_{n\in\wedge} I_n \) belongs to $\mathcal{D}$, which will establish \( I \) as an upper bound for the chain. Indeed, by (b), there exist some \( k \in \wedge \) and \( s \in S \) such that \( s I_n \subseteq I_k \) for every \( n \in \wedge \). Consequently, we obtain $sI = s \left(\bigcup\limits_{n\in\wedge} I_n \right) \subseteq I_k$. Since $\mathcal{D}$ is $S$-saturated, it follows that \( I \in \mathcal{D} \), as required. Applying Zorn's lemma, we conclude that $\mathcal{D}$ has a maximal element.
			
			\item $(c)\implies(d)$. Let $\mathcal{D}$ be a nonempty set of nonnil ideals of $R$. Consider the family $\mathcal{D}^{S}$ of all nonnil ideals $L \subseteq R$ such that there exist some $s \in S$ and $L_0 \in \mathcal{D}$ with $sL \subseteq L_0$. Clearly, $\mathcal{D}\subseteq \mathcal{D}^{S}$, so $\mathcal{D}^{S}\neq\emptyset$. It is straightforward to see that $\mathcal{D}^{S}$ is $S$-saturated. Thus, by (c) implies that $\mathcal{D}^{S}$ has a maximal element $K\in\mathcal{D}^{S}$. Fix $s \in S$ and $J \in \mathcal{D}$ such that $sK\subseteq J$. Now, we claim that $J$ is an $S$-maximal element of $\mathcal{D}$; specifically, given $L\in \mathcal{D}$ with $J\subseteq L$, we show that $sL \subseteq J$. Note that $K + L$ satisfies $s(K + L)= sK + sL \subseteq J + L \subseteq L$, so that $K + L \in \mathcal{D}^{S}$. Also, if $(K+L)\subseteq Nil(R)$, then $K\subseteq Nil(R)$, which is not possible since $K$ is nonnil ideal of $R$. Thus $K+L$ is a nonnil ideal of $R$. Therefore maximality of $K$ implies $K = K + L$, so that $L \subseteq K$. 
			But then $sL \subseteq sK \subseteq J$, as desired.
			
			\item $(d)\implies(a)$. Let $I$ be a nonnil ideal of $R$. Then we will show that $I$ is $S$-finite. Let $\mathcal{D}$ be the family of nonnil finitely generated ideal $J$ of $R$ such that $J\subseteq I$. Choose $x\in I\setminus Nil(R)$. Then $J=(x)\subseteq I$, and $J\nsubseteq Nil(R)$. This implies that $J\in\mathcal{D}$, and so $\mathcal{D}$ is nonempty. Then $\mathcal{D}$ has an $S$-maximal element $K\in\mathcal{D}$. Fixing $x \in I$, take a finitely generated ideal of the form $L= K + xR$. Since $K\subseteq I$ and $x\in I$, so $L\subseteq I$. Consequently, $L\in\mathcal{D}$ such that $K\subseteq L$. This implies that there exists $s\in S$ such that $sL\subseteq K$; in particular, $sx\in K$. This verifies $sI\subseteq K$, so that $I$ is $S$-finite. It follows that $R$ is nonnil-$S$-Noetherian.
		\end{enumerate}
		
	\end{proof}
	
	\noindent	
	The following lemma provides a connection between the concepts of nonnil $S$-irreducible ideals and $S$-primary ideals.
	
	\begin{lem}\label{sp}
		\noindent
		Let R be a nonnil-$S$-Noetherian ring. Then every nonnil $S$-irreducible ideal of $R$ is $S$-primary.
	\end{lem}
	\begin{proof}
		Suppose $Q$ is a nonnil $S$-irreducible ideal of $R$. Let $a,b\in R$ such that $ab\in Q$ and $sb\notin Q$ for all $s\in S$. Our aim is to show that there exists $t\in S$ such that $ta\in rad(Q)$.
		Consider  $A_n=\{x\in R\hspace{0.2cm}|\hspace{0.1cm} a^{n}x\in Q \}$ for $n\in\mathbb{N}$. Since $Q$ is nonnil, there exists $\alpha\in Q\setminus Nil(R)$.	Clearly, $\alpha\in A_n$ for each $n$ but $\alpha\notin Nil(R)$. Consequently, each $A_{n}$ is a nonnil ideal of $R$ and $A_1\subseteq A_2\subseteq A_3\subseteq \cdots$ is an increasing chain of ideals of $R$. Since $R$ is a nonnil-$S$-Noetherian, by Lemma \ref{nonnil}, this chain is $S$-stationary, i.e., there exist $k\in \mathbb{N}$ and $s\in S$ such that $sA_n\subseteq A_k$ for all $n\geq k$. Consider the two ideals $I=(a^{k}) +\hspace{0.1cm} Q$ and $J=(b) +\hspace{0.1cm} Q$ of $R$. Then  $Q\subseteq I\cap J$. For the reverse containment, let $y\in I\cap J$. Write  $y=a^{k}z+q$ for some $z\in R$ and $q\in Q$. Since $ab\in Q$, $aJ\subseteq Q$; whence $ay\in Q$. Now, $a^{k+1}z=a(a^{k}z)=a(y-q)\in Q$. This implies that $z\in A_{k+1}$, and so  $sz\in sA_{k+1}\subseteq A_k$. Consequently, $a^{k}sz\in Q $ which implies that  $a^{k}sz +sq=sy\in Q$. Thus we have $s(I\cap J)\subseteq Q\subseteq I\cap J$. This implies that there exists $s'\in S$ such that either $ss'I\subseteq Q$ or $ss'J\subseteq Q$ since $Q$ is  $S$-irreducible. If $ss'J\subseteq Q$, then $ss'b\in Q$  which is not possible. Therefore $ss'I\subseteq Q$ which implies that $ss'a^{k}\in Q$. Put $t=ss'\in S$. Then $(ta)^{k} \in Q$ and hence  $ta\in rad(Q)$, as desired.
	\end{proof}
	
	Now, we are in a position to represent nonnil-$S$-Noetherian rings as a class of examples of nonnil-$S$-Laskerian rings.
	\begin{theorem}\label{on}
		Every nonnil-$S$-Noetherian ring is a nonnil-$S$-Laskerian ring.
	\end{theorem}
	
	\begin{proof}
		\noindent
		Suppose $R$ is a nonnil-$S$-Noetherian ring. Let $E$ be the collection of nonnil ideals of $R$ which are disjoint with $S$ and can not be written as a finite intersection of $S$-primary ideals. We wish to show $E=\emptyset$. On the contrary suppose  $E\neq\emptyset$.  Since $R$ is a nonnil-$S$-Noetherian ring, by Lemma \ref{nonnil}, there exists an $S$-maximal element in $E$, say $I$. Evidently, $I$ is not an $S$-primary ideal, by Lemma \ref{sp}, $I$ is not an $S$-irreducible ideal, and so  $I$ is not an irreducible ideal. This implies that  $I=J\cap K$ for some ideals $J$ and $K$ of $R$ with $I\neq J$ and $I\neq K$. Since $I$ is nonnil, there exists $\alpha\in I\setminus Nil(R)$. This implies that $\alpha\in J$ and $\alpha\in K$ but $\alpha\notin Nil(R)$. Consequently, $J$ and $K$ both are nonnil ideals of $R$. Since $I$ is not $S$-irreducible, $sJ\nsubseteq I$ and $sK\nsubseteq I$ for all $s\in S$. Now, we claim that $J$, $K\notin E$. For this, if $J$ (respectively, $K$) belongs to $E$, then since $I$ is an $S$-maximal element of $E$ and $I\subset J$ (respectively, $I\subset K$), there exists $s'$ (respectively, $s''$) from $S$ such that $s'J \subseteq I$ (respectively, $s''K\subseteq I$). This is not possible, as $I$ is not $S$-irreducible. Therefore $J, K\notin E$. Also, if $J\cap S\neq\emptyset$ $(respectively, K\cap S\neq\emptyset)$, then there exist $s_{1}\in J\cap S$  $(respectively, s_{2}\in K\cap S)$. This implies that $s's_{1}\in s'J\subseteq I$ $(respectively, s''s_{2}\in s''K\subseteq I)$, which is a contradiction because $I$ disjoint with $S$. Thus $J$ and $K$ are also disjoint with $S$. This implies that $J$ and $K$ can be written as a finite intersection of $S$-primary ideals. Consequently, $I$ can also be written as a finite intersection of $S$-primary ideals since  $I=J\cap K$, a contradiction as $I\in E$. Thus $E=\emptyset$, i.e., every proper ideal of $R$ disjoint with $S$ can be written as a finite intersection of $S$-primary ideals. Hence $R$ is a nonnil-$S$-Laskerian ring.
	\end{proof}
	
\end{document}